\pdfoutput=1
\documentclass{amsart}
\usepackage{todonotes}
\usepackage{graphicx, amsmath, amssymb, amsthm}
\usepackage{tikz-cd}
\usepackage{spectralsequences}
\usepackage{hyperref} 
\usepackage{adjustbox}
\usepackage[backend=biber, style=numeric, maxbibnames=99, doi=true,url=false,
    giveninits=true, hyperref]{biblatex}
\renewbibmacro{in:}{}

\newcommand{\ip}[1]{\langle#1\rangle}

\newcommand{\tens}[1]{%
  \mathbin{\mathop{\otimes}\limits_{#1}}%
}

\newcommand{\Z}{{\mathbb  Z}}

\newcommand{\R}{{\mathbb R}}
\newcommand{\F}{{\mathbb F}}

\DeclareMathOperator{\Slicegr}{SliceGr}

\newcommand{\timesover}[1]{\underset{#1}{\times}}

\newcommand{\Boxover}[1]{\underset{#1}{\Box}}
\newcommand{\otimesover}[1]{\underset{#1}{\otimes}}

\DeclareMathOperator{\Ext}{Ext}

\DeclareMathOperator{\Spec}{Spec}

\newcommand{\tr}{tr}

\newcommand{\Sp}{\mathcal Sp}

\mathchardef\mhyphen=45

\numberwithin{equation}{section}

\newtheorem{theorem}{Theorem}[section]
\newtheorem{lemma}[theorem]{Lemma}
\newtheorem{corollary}[theorem]{Corollary}

\newtheorem{proposition}[theorem]{Proposition}

\newtheorem*{theorem*}{Theorem}
\newtheorem*{proposition*}{Proposition}

\theoremstyle{remark}
\newtheorem{remark}[theorem]{Remark}

\theoremstyle{definition}

\title{On $MU$-homology of connective models of higher real $K$-theories}
\author{Christian Carrick}
\address{Mathematical Institute, Utrecht University, Utrecht, 3584 CD, the Netherlands}
\email{c.d.carrick@uu.nl}
\thanks{The first author was supported by the National Science Foundation under Grant No. DMS-2401918. The second author was supported by the National Science Foundation under Grant No. DMS-2105019}
\author{Michael A. Hill}
\address{University of Minnesota, Minneapolis, MN 55455}
\email{mahill@umn.edu}

\begin{document}

\begin{abstract}
    We use the slice filtration to study the $MU$-homology of the fixed points of connective models of Lubin--Tate theory studied by Hill--Hopkins--Ravenel and Beaudry--Hill--Shi--Zeng. We show that, unlike their periodic counterparts $EO_n$, the $MU$ homology of $BP^{((G))}\langle m\rangle^G$ usually fails to be even and torsion free. This can only happen when the height $n=m|G|/2$ is less than $3$, and in the edge case $n=2$, we show that this holds for $tmf_0(3)$ but not for $tmf_0(5)$, and we give a complete computation of the $MU_*MU$-comodule algebra $MU_*tmf_0(3)$.
\end{abstract}

\maketitle
\section{Introduction}

Let $E_n$ be a height $n$ Lubin--Tate theory, $G$ a finite subgroup of the corresponding Morava stabilizer group, and $EO_n(G):=E_n^{hG}$ the corresponding \emph{higher real $K$-theory}. The $EO_n(G)$-theories were first defined by Hopkins--Miller, intended to provide more computable approximations of the $K(n)$-local sphere. For example, at $p=2$, taking $G=C_2$ at height $n=1$ gives $KO$, and $G=G_{24}$ (the binary tetrahedral group) at height $n=2$ gives $TMF$, up to $K(n)$-localization.

Hopkins--Miller also observed that these theories are most powerful when they admit good connective models $eo_n(G)$, with $ko$ and $tmf$ being the prototypes. Such connective models should satisfy certain finiteness properties and therefore lend themselves to sparse computations, and they should in particular be accessible from the point of view of both the Adams and Adams--Novikov spectral sequences. 

At the prime \(2\), we have good candidates for the $eo_n$ theories: the $BP^{((G))}\langle m\rangle^{G}$s of Hill--Hopkins--Ravenel (HHR), and HHR showed these theories are computationally accessible via the slice spectral sequence. At height $0$, one has $BP^{((G))}\langle 0\rangle^{G}\simeq H\Z$ for all $G$. At height $1$, the spectrum $BP^{((C_2))}\langle 1\rangle^{C_2}=BP_\R\langle1\rangle^{C_2}$ gives $ko$. At height $2$, a theorem of the second author and Meier \cite{hillmeier} shows that $BP^{((C_2))}\langle 2\rangle^{C_2}=BP_\R\langle2\rangle^{C_2}$ gives $tmf_0(3)$, and it is expected that $BP^{((C_4))}\langle 1\rangle^{C_4}$ gives the spectrum $tmf_0(5)$. The height $4$ theory $BP^{((C_8))}\langle 1\rangle^{C_8}$ is a connective model of the detection spectrum $\Omega$ of HHR used in their solution of the Kervaire invariant one problem.

\begin{theorem}\label{thm1}
    For $|G|>2$ and $m>0$, the $MU$-homology of $BP^{((G))}\langle m\rangle^G$ has torsion classes in odd degrees. When $G=C_2$, the $MU$-homology of $BP_\R\langle m\rangle^{C_2}$ is torsion free and concentrated in even degrees for $m\le 2$ and has torsion classes in odd degrees for $m>2$.
\end{theorem}

In contrast, the first author showed that the $BP$-homology of $EO_n(G)$ is always torsion free and concentrated in even degrees \cite{defect}. 

This theorem marks a qualitative difference in the complexity of these theories when one moves beyond height $2$, reflecting similar observations of the authors and Ravenel in \cite[Theorem 1.11]{CHR}. The computations of \cite{HHR} show that, for $m>0$, there are nonzero classes 
\[\eta_G:=a_{\overline{\rho}_G}N(\overline{r}_1^G)\in\pi_1BP^{((G))}\langle m\rangle^G\] with $2\eta_G=\eta_G^3=0$ of slice filtration $|G|-1$ lying on the slice vanishing line. When $G=C_2$, $\eta_G$ is the usual Hopf map $\eta$, which is sent to zero in $MU$-homology. In contrast, we prove Theorem \ref{thm1} by showing that, for $|G|>2$, $\eta_G$ is nonzero in $MU$-homology and hence has Adams--Novikov filtration zero.

Significant computations with $BP^{((G))}\langle m\rangle^{G}$ have been carried out using the slice spectral sequence. For example, HHR computed the slice spectral sequence of $BP^{((C_4))}\langle 1\rangle^{C_4}$\cite{hhrc4}, and Hill--Shi--Wang--Xu computed the slice spectral sequence of the height 4 theory $BP^{((C_4))}\langle 2\rangle^{C_4}$ \cite{hswx}. Little is known about the Adams spectral sequence and Adams--Novikov spectral sequence of these theories, however, even at height $2$. 

In \cite{CHR}, the authors and Ravenel introduced a family of spectral sequences that leverage the well-understood slice filtration to compute things like $H_*(BP^{((G))}\langle m\rangle^G)$ and $MU_*BP^{((G))}\langle m\rangle^G$, which are the input to the Adams spectral sequence and Adams--Novikov spectral sequence respectively. The authors  computed the homology groups $H_*(BP^{((C_2))}\langle m\rangle^{C_2})$ for $m\le 3$. This resulted for example in an understanding of the Adams spectral sequence for $tmf_0(3)$ and revealed also a divergence of complexity and behavior beginning at height 3.

In this paper, we use the setup of \cite{CHR} to carry out a similar analysis on the Adams--Novikov spectral sequence side. This is especially interesting because the slice spectral sequence of $BP^{((G))}\langle m\rangle$ behaves in many ways like an Adams--Novikov spectral sequence. For example, the slice spectral sequence of $BP^{((C_2))}\langle 1\rangle=BP_\R\langle 1\rangle$ is isomorphic to the Adams--Novikov spectral sequence of $ko$, as one may observe directly. The authors were thus motivated to see how generally such an identification holds.

Our findings show that, beyond the degenerate height zero case, this holds only for $ko$. The zero line in the slice spectral sequence for $BP^{((G))}\langle m\rangle$ is concentrated in even degrees, so Theorem \ref{thm1} shows this cannot agree with the Adams--Novikov spectral sequence for $BP^{((G))}\langle m\rangle^G$ whenever $|G|>2$ or $m>2$. In the remaining edge case $|G|=m=2$, we give a full computation.

\begin{theorem}\label{thm2}
    There is an isomorphism of $MU_*MU$-comodule algebras 
    \[MU_*tmf_0(3)\to (MU_*MU\tens{MU_*}A)\Boxover{\Gamma}M\]
    where $(A,\Gamma)=(\Z[a_1,a_3],A[s,t]/\sim)$ is the $2$-local Weierstrass algebroid of \cite[Section 7]{bauer}, and \[M=A\{1,e_4,e_6,e_8,e_{10},e_{12}\}\]
    is an $(A,\Gamma)$-comodule algebra, given by explicit formulas in Theorem \ref{thm:tmf03computation}.
\end{theorem}

\begin{remark}
    A low dimensional $\mathrm{Ext}$ calculation using Theorem \ref{thm2} shows that the Hopf map $\nu$ is detected in Adams--Novikov spectral sequence filtration 1 in $tmf_0(3)$, whereas it is detected in slice filtration 3, so these spectral sequences do not agree. This points to a surprising difference between the connective model $tmf_0(3)$ and its projective and periodic variants $Tmf_0(3)$ and $TMF_0(3)$, both of which detect $\nu$ in Adams--Novikov spectral sequence filtration 1.
\end{remark}

\subsection*{Summary} We begin in Section \ref{sec:2} by recalling the necessary ingredients regarding the slice spectral sequence with coefficients from \cite{CHR} to perform our computations. In Section \ref{sec:3}, we prove the $|G|>2$ cases of Theorem \ref{thm1}. This is easier than the $C_2$ cases since the torsion appears already in degree 1, whereas we need to look in degree 9 in the $C_2$-cases. We verify the latter cases in Section \ref{sec:4} and compute $BP_*ko$. We finish in Section \ref{sec:5} by computing $BP_*tmf_0(3)$.

We work $2$-locally throughout. In particular, we compute $BP$-homology in practice, and the results on $MU$-homology stated in the introduction are easy consequences of the $2$-local splitting of $MU$ as a wedge of shifts of $BP$.

\subsection*{Acknowledgments} 
The authors are grateful to Hausdorff Research Institute for Mathematics in Bonn funded by the Deutsche Forschungsgemeinschaft (DFG, German Research Foundation) under Germany's Excellence Strategy - EXC-2047/1 - 390685813 and to the organisers of the trimester program \emph{Spectral Methods in Algebra, Geometry, and Topology} 2022. 
The first-named author would like to thank Lennart Meier and Jack Davies for helpful comments and discussion.

\section{The $BP$-based slice spectral sequence}\label{sec:2}
In \cite{CHR}, the authors and Ravenel developed a generalized slice filtration to compute the $K$ homology of $E$ via smashing the slice tower of $E$ with $K$, for $K,E\in \Sp^G$. This produces a spectral sequence of signature
\[E_2=\pi_\bigstar(K\otimes\Slicegr(E))\implies K_\bigstar E\]
Setting $K=i_*BP$ (a.k.a. the inflation of $BP$) and $E=BP^{((G))}\langle m\rangle$, the corresponding spectral sequence converges to $i_*BP_\bigstar BP^{((G))}\langle m\rangle$, and in integer degrees, we have an isomorphism
\[(i_*BP)_*BP^{((G))}\langle m\rangle\cong BP_*BP^{((G))}\langle m\rangle^G\]
Collecting results from \cite[Section 2.1]{CHR}, we have the following convenient properties of this spectral sequence.

\begin{proposition}\label{prop:ssvanishinglines}
    The $i_*BP$-based slice spectral sequence of $BP^{((G))}\langle m\rangle$ 
    \[E_2^{s,t}=\pi_{t-s}(i_*BP\otimes\Slicegr(BP^{((G))}\langle m\rangle))\implies\pi_{t-s}(i_*BP\otimes BP^{((G))}\langle m\rangle)\]
    converges strongly and is a right half-plane spectral sequence concentrated between the lines $s=(|G|-1)(t-s)$ and $s=-(t-s)$ from $E_2$ on. It is a spectral sequence of $BP_*BP$-comodules, and if $BP^{((G))}\langle m\rangle$ is a ring spectrum, one of $BP_*BP$-comodule algebras.
\end{proposition}

In contrast with the $i_*H\F_2$-based slice spectral sequence studied in \cite{CHR}, the $E_2$-page of the $i_*BP$-based version has a more straightforward structure. This follows from the slice theorem of Hill--Hopkins--Ravenel.

\begin{proposition}\label{prop:e2bpgm}
    The $E_2$-page of the $i_*BP$-based slice spectral sequence of $BP^{((G))}\langle m\rangle$ is given by
    \[A_m\otimes\Z[t_1,t_2,\ldots]\]
    where $A_m$ is the $E_2$-page of the ordinary slice spectral sequence of $BP^{((G))}\langle m\rangle$, and $|t_i|=(2^i-1,1-2^i)$.
\end{proposition}
\begin{proof}
    The $E_2$-page in question is given by 
    \[\pi^G_*(i_*BP\otimes \Slicegr BP^{((G))}\langle m\rangle)=BP_*\Slicegr BP^{((G))}\langle m\rangle^G\]
    By the HHR slice theorem \cite[Theorem 6.1]{HHR}, $\Slicegr BP^{((G))}\langle m\rangle$ is an $H\underline{\Z}$-module, so that
    applying fixed points, one has an $H\Z$-module. In particular, the slice associated graded of $BP^{((G))}\langle m\rangle^G$ is complex-orientable, so the result follows from the fact that for any complex-orientable spectrum $R$, one has an isomorphism
    \[BP_*R\cong R_*\otimes\Z[t_1,t_2,\ldots]\]
    (see \cite[4.1.7]{ravgreen}).
 \end{proof}

\begin{remark}
    As in \cite{CHR}, we work only in the positive cone for the $G=C_2$ case. Working in the positive cone is equivalent to working motivically with the motivic lifts $BPGL\ip{n}$ of $BP_\R\ip{n}$. However, we stick to equivariant language in this article and refer to the motivic setting only to justify working in the positive cone. We will make no further mention of this and work in the positive cone throughout.
\end{remark}

\subsection{Comparison with the $i_*H\F_2$-based slice spectral sequence for $G=C_2$} The truncation map $BP\to H\F_2$ induces a map from the $i_*BP$-based slice spectral sequence of $BP_\R\ip{n}$ to the $i_*H\F_2$-based slice spectral sequence of $BP_\R\ip{n}$, which we use to deduce differentials in the former from the latter.

\begin{proposition}\label{prop:morphismofHSSS}
    The morphism of slice spectral sequences for $BP_\R\ip{n}$ induced by $i_*BP\to i_*H\F_2$ is given on $E_2$ by the map
    \[\Z[a_\sigma,u_{2\sigma},\overline{v}_1,\ldots,\overline{v}_n,t_1,t_2,\ldots]/(2a_\sigma)\to (\mathcal{A}_*\Boxover{\mathcal{A}(0)_*}\F_2)[a_\sigma,x_1,\overline{v}_1,\ldots,\overline{v}_n]\]
    sending
    \begin{align*}
        a_\sigma&\mapsto a_\sigma,&
        u_{2\sigma}&\mapsto x_1^2+a_\sigma^2\zeta_1^2,&\overline{v_i}&\mapsto \overline{v_i},&
        t_i&\mapsto \zeta_i^2.
    \end{align*}
\end{proposition}
\begin{proof}
    The description of the $E_2$-pages follow from Proposition \ref{prop:e2bpgm} and \cite[Corollary 4.5]{CHR}, respectively. The claim for $a_\sigma$ and $\overline{v_i}$ follow from the fact that both spectral sequences are built from the slice tower of $BP_\R\ip{n}$, the claim for $u_{2\sigma}$ follows from \cite[Corollary 4.4]{CHR}, and the claim for $t_i$ follows from the fact that the map $BP_*H\F_2\to {H\F_2}_*H\F_2=\mathcal{A_*}$ sends $t_i\mapsto\zeta_i^2$.
\end{proof}

The following are immediate corollaries and allow us to lift differentials from the $i_*H\F_2$ case to the $i_*BP$ case.

\begin{corollary}\label{cor:injectiveonasigma}
    The morphism of slice spectral sequences for $BP_\R\ip{n}$ induced by $i_*BP\to i_*H\F_2$ is injective on the ideal generated by $a_\sigma$ on $E_2$.
\end{corollary}

\begin{corollary}\label{cor:t_isupportsdiff}
    In the $i_*BP$-based slice spectral sequence for $BP_\R\ip{n}$, for $i\le n$ and $j<i+n-1$, the class $t_i^{2^j}$ supports a nontrivial $d_r$ for some $r\le 2^{i+j+1}-1$.
\end{corollary}
\begin{proof}
    This follows by naturality, Proposition \ref{prop:morphismofHSSS}, and \cite[Theorem 1.7]{CHR}.
\end{proof}

\section{Odd classes in $BP_*BP^{((G))}\langle m\rangle^G$ for $|G|>2$}\label{sec:3}

To identify a nonzero class in odd degrees in $BP_*BP^{((G))}\langle m\rangle^G$, we will need a few facts about the slice spectral sequence of $BP^{((G))}\langle m\rangle$ in stems $t-s\le 3$. These can be deduced from \cite[Section 5]{eta3}, so we give here a brief explanation.

\begin{lemma}\label{lemma:lowslicebpgm}
    The following hold in the slice spectral sequence of $BP^{((G))}\langle m\rangle$ whenever $|G|>2$ and $m>0$:
    \begin{enumerate}
        \item The $E_2$-page vanishes above the line $s=(|G|-1)(t-s)$ and whenever $t$ is odd.
        \item The bidegree $(t-s,s)$ is generated by transfers from $C_2$ whenever $t=6$.
        \item In bidegree $(t-s,s)=(1,1)$ there is a class $\eta:=tr_{2}^{|G|}(a_{\sigma_2}\overline{r}_1^G)$ that detects $\eta\in\pi_1\mathbb{S}$.
        \item In bidegree $(t-s,s)=(3,1)$ there is a class 
        \[\eta_{C_4}:=\begin{cases}a_{\overline{\rho}_4}N_{2}^4(\overline{r}_1^G)&|G|=4\\tr_{4}^{|G|}(a_{\overline{\rho}_4}N_{2}^4(\overline{r}_1^G))&|G|>4
        \end{cases}\]
        whose square is nonzero on $E_2$.
    \end{enumerate}
\end{lemma}
\begin{proof}
    For claim (1), the vanishing line is \cite[Proposition 4.40]{HHR}, and the odd vanishing is the odd vanishing in the HHR slice theorem. Claim (2) follows from the slice theorem: each generator of $\pi_6^uBP^{((G))}\langle m\rangle$ has stabilizer $C_2$, hence each summand of $P^6_6$ is induced from $C_2$. 
    
    For claim $(3)$, it follows from \cite[Theorem 1.2]{LSWX} that $BP^{((G))}\langle m\rangle^{C_2}$ detects $\eta$ in slice filtration 1. Since the restriction map $BP^{((G))}\langle m\rangle^{G}\to BP^{((G))}\langle m\rangle^{C_2}$ factors the Hurewicz map of $BP^{((G))}\langle m\rangle^{G}\to BP^{((G))}\langle m\rangle^{C_2}$, and bidegree $(1,0)$ vanishes in the slice spectral sequence for $BP^{((G))}\langle m\rangle$ by claim (1), naturality of the slice spectral sequence implies that $\eta$ must be detected in the slice spectral sequence for $BP^{((G))}\langle m\rangle$ in bidegree $(1,1)$, which is generated by the claimed transfer. Claim (4) follows by analyzing the slices in low degrees and using the Frobenius formula for products of transfers in a Green functor (see \cite[Proposition 5.7]{eta3}).
\end{proof}

\begin{theorem}
    For $m>0$, the class $\eta_{C_4}\in\pi_1(BP^{((G))}\langle m\rangle^G)$ is not in the kernel of the map
    \[\pi_*(BP^{((G))}\langle m\rangle^G)\to BP_*(BP^{((G))}\langle m\rangle^G)\]
    In particular, the $BP$-homology of $BP^{((G))}\langle m\rangle^G$ has a $2$-torsion class in odd degrees that is in the kernel of the restriction map.
\end{theorem}
\begin{proof}
    We will show that in the $i_*BP$-based slice spectral sequence of Proposition \ref{prop:e2bpgm}, the permanent cycle $\eta_{C_4}$ cannot be hit by a differential. For degree reasons, the only possible differential is a $d_5$ with source in bidegree $(2,-2)$, which is generated on $E_2$ by $t_1$. We illustrate the situation with $BP^{((C_4))}\langle 1 \rangle$ in Figure \ref{bpc41}.

    In fact, since $\eta=0\in BP_*$, by item (3) of Lemma \ref{lemma:lowslicebpgm}, the class $\eta$ in bidegree $(1,1)$ must be hit by a differential. The only possibility is $d_3(t_1)=\eta$, hence the bidegree $(2,-2)$ on $E_5$ is generated by $[2t_1]$. We claim also that $(\eta_{C_4})^2$ remains nonzero on $E_5$. This finishes the proof as the last possible differential killing $\eta_{C_4}$ is $d_5([2t_1])=\eta_{C_4}$, but since $[2t_1]\cdot\eta_{C_4}=0$ as $\eta_{C_4}$ is $2$-torsion (because $a_{\overline{\rho}_4}$ is $2$-torsion), this would imply $0=d_5([2t_1]\cdot\eta_{C_4})=(\eta_{C_4})^2$, a contradiction.

    If $(\eta_{C_4})^2$ were $0$ on $E_5$, by item (4) of the lemma, it must the target of a differential of length $<5$, which must be a $d_3$ by evenness. The source of such a differential is in bidegree $(3,3)$, which consists of transfers from $C_2$ by item (2) of the lemma, all of which are permanent cycles in view of the vanishing line $t-s=s$ in the $C_2$-slice spectral sequence.
\end{proof}

\begin{sseqdata}[ name = bpc41, Adams grading, classes = {fill, show name=below},
grid = go, xrange ={0}{3},yrange={-3}{6},xscale=1,yscale=.5,x tick step =2, y tick step =2,run off differentials = {->},struct lines = black ]
\class[name = 1,rectangle](0,0)
\class[name = \eta,red](1,1)
\DoUntilOutOfBoundsThenNMore{2}{
    \class[red](\lastx+1,\lasty+1)
    \structline
}
\structline(0,0)(1,1)

\class[name = b_1,rectangle](2,-2)
\DoUntilOutOfBoundsThenNMore{3}{
    \d3
    \class[red](\lastx+1,\lasty+1)
    \structline
}

\class[name = \nu](3,1)
\class[name=\eta_{C_4}](1,3)
\class(2,6)
\structline(0,0)(1,3)
\structline(1,3)(2,6)
\class[name=2\nu,red](3,3)
\structline[dashed](3,1)(3,3,2)
\class(3,1)
\class(4,4)
\structline(2,-2)(3,1,2)
\structline(3,1,2)(4,4)

\end{sseqdata}
\begin{figure}[!htbp]
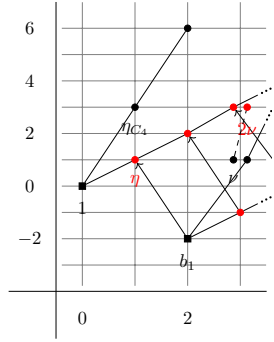

\adjustbox{scale=0.7,center}{
\printpage[name = bpc41]
}
\caption{The $i_*BP$-based HSSS for $BP^{((C_4))}\langle 1\rangle$ with transfers in red.}
\label{bpc41}
\end{figure}

\begin{remark}
    In fact, a more careful but analogous argument shows that all elements in $\pi_1BP^{((G))}\langle m\rangle^G$ of slice filtration $>1$ (all of which are of the form $\eta_H$ for $H\subset G$) survive the map to $BP$-homology.
\end{remark}

\section{The $BP$-homology of $BP_\R\langle n\rangle^{C_2}$}\label{sec:4}
We run the $i_*BP$-based slice spectral sequence for $BP_\R\langle n\rangle$. For $n=1$ and $2$, this behaves as expected and becomes torsion free and even on the $E_4$ and $E_8$ pages respectively. For $n\ge 3$, this pattern is interrupted by a class that appears in the $9$ stem in filtration $5$.

\subsection{The $E_3$ page and $BP_*ko$}
Our first differentials appear on the $E_3$-page where we see the differential for the ordinary slice spectral sequence of $BP_\R\ip{n}$ as well as the $d_3$ of Corollary \ref{cor:t_isupportsdiff}.

\begin{proposition}\label{prop:d3}
    In the $i_*BP$-based slice spectral sequence for $BP_\R\ip{n}$ for $n\ge1$, up to use of the Leibniz rule, the only nonzero $d_3$ differentials are 
    \begin{align*}
        d_3(u_{2\sigma})&=a_\sigma^3\overline{v}_1&d_3(b_1)&=a_\sigma\overline{v}_1.
    \end{align*}
\end{proposition}
\begin{proof}
    The first differential follows by naturality and the slice differentials theorem \cite[Theorem 9.9]{HHR}. The second follows by Corollary \ref{cor:t_isupportsdiff}: $b_1$ must thereby support a nonzero $d_3$, and the claimed target is the only possibility.

    The generators $t_i$ for $i>1$ cannot support a $d_3$ for degree reasons. Any possible differential $d_3(t_i)$ can be written as a sum of monomials of the form $a_\sigma^au_{2\sigma}^bp(\overline{v_i})q(t_i)$,
    where $p$ is a monomial in the $\overline{v_i}$'s and $q$ is a monomial in the $t_i$'s. Each such monomial must live in an odd stem, and therefore must have $a>0$. By Corollary \ref{cor:injectiveonasigma}, however, the morphism of spectral sequences induced by $i_*BP\to i_*H\F_2$ is injective on \(E_3\) when restricted to the ideal generated by $a_\sigma$. Since $t_i$ does not support a differential in the $i_*H\F_2$-based slice spectral sequence for $i>1$ by \cite[Theorem 1.7]{CHR}, the same is true in the $i_*BP$-based slice spectral sequence .
\end{proof}

\begin{corollary}\label{prop:E4page}
    We have
    \[
        E_4=\Z[a_\sigma,w_{2\sigma},m_1,\overline{v}_1,\dots,\overline{v}_n,t_1^2,t_2,\ldots]/(2a_\sigma,a_\sigma\overline{v}_1, a_\sigma m_1,m_1^2-4t_1^2).
    \]
    where $m_1:=[2t_1]$ and \(w_{2\sigma}=u_{2\sigma}-a_\sigma^2 t_1\). 
\end{corollary}

Setting $n=1$ and $\Delta:=w_{2\sigma}\overline{v_1}^2$, we see that in integer degrees
\[E_4=\Z[\Delta,m_1,t_1^2,t_2,\ldots]/(m_1^2=4t_1^2)\]
which is already concentrated in even stems, and therefore collapses. There is, however, a nontrivial multiplicative extension, which may be resolved using the Mackey structure. Indeed, since $a_\sigma$ does not contribute to the above $E_4$-page, the restriction map $BP_*ko\to BP_*ku$ is an injection as it induces an injection on slice associated graded. Using the $C_2$ action on the underlying $BP$-homology, one can deduce the following. Since it is a well-known computation, we omit the details.

\begin{proposition}\label{prop:T(m)ofko}
    There is a lift $c$ of $t_1\gamma(t_1)$ along the restriction map $BP_*ko\to BP_*ku$. The resulting ring map 
    \[f:\Z[b,c,t_2,\ldots]\to BP_*ko\]
    sending $b\mapsto m_1$ and $c\mapsto c$
    is an isomorphism. It sends $b^2-4c$ to $\Delta$.
\end{proposition}

\subsection{The $E_7$-page}
We now set the height $n
\ge 2$. There are no $d_r$-differentials in the $i_*BP$-based slice spectral sequence for $BP_\R\ip{n}$ for $3<r<7$ for degree reasons. The $d_7$'s can once again be determined by comparison with the $i_*H\F_2$-based slice spectral sequence.

\begin{proposition}\label{prop:d7}
In the $i_*BP$-based slice spectral sequence for $BP_\R\ip{n}$ for $n\ge2$, up to use of the Leibniz rule, the only nonzero $d_7$ differentials are 
\begin{align*}
    d_7(t_1^2)&=a_\sigma^3\overline{v}_{2}&d_7(t_2)&=a_\sigma w_{2\sigma}\overline{v}_{2}
\end{align*}
\end{proposition}
\begin{proof}
    The claimed nonzero differentials follow from Corollary \ref{cor:t_isupportsdiff}: $t_1^2$ and $t_2$ must thereby support a $d_7$, and these are the only possibilities for degree reasons.

    A slight modification of the proof of Proposition \ref{prop:d3} shows that $d_7(t_i)=0$ for $i>2$. Since the target must lie in an odd stem, Corollary \ref{prop:E4page} implies it must be divisible by $a_\sigma$, and that the map of spectral sequences induced by $i_*BP\to i_*H\F_2$ is still an injection on $E_4$ when restricted to the ideal $(a_\sigma)$.
\end{proof}

Taking homology with respect to these differentials, we see that $E_8$ is generated as an algebra by the classes
\begin{itemize}
\item $a_\sigma$, $w_{2\sigma}$, $\overline{v}_{1},\ldots,\overline{v}_{n}$, $t_1^4,t_2^2,t_3,\ldots$, $[w_{2\sigma}t_1^2+a_\sigma^2t_2]$
\item $[2t_1],[2t_2],[2t_1^2],[2t_1^2t_2],[2t_1^3],[2t_1t_2],[2t_1^3t_2]$, $[\overline{v}_{1}t_1^2],[\overline{v}_{1}t_2],[\overline{v}_{1}t_1^2t_2]$
\end{itemize}

Setting $n=2$, the relations imposed by the differentials imply that $a_\sigma$ does not contribute to integer degrees. We conclude the following.

\begin{proposition}\label{prop:T(m)tmf03}
    The $BP$-homology of $tmf_0(3)$ is torsion-free and concentrated in even degrees, and the restriction map
    \[BP_*tmf_0(3)\to BP_*tmf_1(3)\]
    is an injection
\end{proposition}
\begin{proof}
    All claims may be checked on $BP$ homology and therefore on the $E_\infty$ page of the $i_*BP$-based slice spectral sequence and follow from the fact that $a_\sigma$ does not contribute to integer degrees on $E_8$. Since $E_8$ is then concentrated in even stems, $E_8=E_\infty$.
\end{proof}

\subsection{The $n>2$ case and odd classes}
Setting our height $n>2$, the computations thus far already allow us to locate a nonzero $2$-torsion class in $BP_9\big(BP_\R\ip{n}^{C_2}\big)$.

\begin{proposition}\label{prop:oddclass}
The permanent cycle $a_\sigma^5w_{2\sigma}\overline{v}_3$ in bidegree $(9,5)$ of the $i_*BP$-based slice spectral sequence of $BP_\R\ip{n}$ for $n>2$ survives to the $E_\infty$ page.    
\end{proposition}
\begin{proof}
    By Propositions \ref{prop:d3} and \ref{prop:d7}, if $a_\sigma^5w_{2\sigma}\overline{v}_3$ does not survive to $E_\infty$, there must be a differential $d_r(x)=a_\sigma^5w_{2\sigma}\overline{v}_3$ for $r>7$. The description of the $E_4$-page from Corollary \ref{prop:E4page} implies that the only nonzero bidegrees of the form $(10,5-r)$ for $r>7$ are
    \[E_4^{10,-6}=\Z\{m_1t_1^2w_{2\sigma}\overline{v}_1^2,t_2w_{2\sigma}\overline{v}_1^2,m_1^3w_{2\sigma}\overline{v}_1^2\}\]
    and
    \[E_4^{10,-10}=\Z\{m_1t_1^4,m_1^3t_1^2,m_1^2t_2,m_1^5,t_1^2t_2\}\]
    Since $m_1$ is a transfer class, all possible sources are permanent cycles except $t_2w_{2\sigma}\overline{v}_1^2$ and $t_1^2t_2$. Both of these classes support $d_7$'s by Proposition \ref{prop:d7}, so that, on $E_8$, the above bidegrees are generated by the transfers listed above and the transfers $[2t_2w_{2\sigma}\overline{v}_1^2]$ and $[2t_1^2t_2]$.
\end{proof}

\section{The $BP$ homology of $tmf_0(3)$}\label{sec:5}
We know from Proposition \ref{prop:T(m)tmf03} that the restriction map 
\[BP_*tmf_0(3)\to BP_*tmf_1(3)\]
is an injection of $BP_*BP$-comodule algebras, and the comodule algebra structure of $BP_*tmf_1(3)$ follows from the fact that $tmf_1(3)$ is complex-orientable. To streamline this calculation, however, we compute a relative version of $BP_*tmf_0(3)$ that carries the same information, since $tmf_0(3)$ is a $tmf$-module. In the following let $(A,\Gamma)=(\Z[a_1,a_3],A[s,t]/\sim)$ be the $2$-local Weierstrass algebroid from \cite[Section 7]{bauer}.

\begin{proposition}\label{prop:cotensor}
    Let $F$ be a $2$-local compact $tmf$-module. There is an isomorphism of $MU_*MU$-comodules
    \begin{equation}\label{eq:comparisonmap}
        MU_*F\to (MU_*MU\tens{MU_*}A)\Boxover{\Gamma}\pi_*(tmf_1(3)\tens{tmf}F)
    \end{equation}
    If $F$ is a $tmf$-algebra, the isomorphism is of $MU_*MU$-comodule algebras.
\end{proposition}
\begin{proof}
    The complex orientation $MU\to tmf_{1}(3)$ induces a map
    \[
        \pi_*(MU\otimes F)\to\pi_*(tmf_1(3)\tens{tmf}F)
    \]
    of $(A,\Gamma)$-comodules where the former has the corestricted comodule structure via the (surjective) map $(MU_*,MU_*MU)\to(A,\Gamma)$ of Hopf algebroids. This induces the adjoint map (\ref{eq:comparisonmap}) of $(MU_*,MU_*MU)$-comodules. By the Milnor--Moore theorem \cite[A1.1.19]{ravgreen}, $MU_*MU\tens{MU_*}A$ is a cofree $(A,\Gamma)$-comodule, so that the right hand side of (\ref{eq:comparisonmap}) sends cofiber sequences in $F$ to exact sequences, and it commutes with finite coproducts since the cotensor is a right adjoint. The map (\ref{eq:comparisonmap}) is therefore a morphism of homology theories on compact $tmf$-modules, and it now suffices to check the case $F=tmf$.

    In fact, the equalizer defining the right hand side of (\ref{eq:comparisonmap}) in the case $F=tmf$ comes from pulling back the cover $\Spec(A)\to\mathcal{M}_{cub}$ to
    \[
        \Spec(MU_*tmf)=\mathcal{M}_{cub}\timesover{\mathcal{M}_{FG}(1)}\Spec(MU_*MU)
    \]
\end{proof}

We are therefore reduced to computing the $(A,\Gamma)$-comodule algebra given by the homotopy groups of  
\[R_0(3):=tmf_1(3)\tens{tmf}tmf_0(3)\]
In fact, since the isomorphism of Proposition \ref{prop:cotensor} is natural in $tmf$-modules, we deduce the following.

\begin{corollary}\label{cor:relativeresinjective}
    The restriction map
    \[
        \pi_*(R_0(3))\to\pi_*(tmf_1(3)\tens{tmf}tmf_1(3))=\Gamma
    \]
    is injective.
\end{corollary}
\begin{proof}
    This follows from Propositions \ref{prop:cotensor} and \ref{prop:T(m)tmf03} using that $MU_*MU\otimes_{MU_*}A$ is a faithfully flat $A$-module and cofree $(A,\Gamma)$-comodule.
\end{proof}

This will allow us to perform our computations inside $\Gamma$. 

\subsection{The relative Adams spectral sequence for $R_0(3)$} 
The authors along with Ravenel showed that there is a splitting of $tmf$-modules $tmf_0(3)\simeq tmf\otimes X$ for $X$ a certain $10$-cell complex \cite{CHR}. This may be used to run the Adams spectral sequence of $R_0(3)$ relative to the ring map $tmf\to H\F_2$ \cite{BakerLazarev}. This takes the form
\[
    E_2=\Ext_{\mathcal{A}(2)_*}\big(\F_2,H_*^{tmf}(R_0(3))\big)\implies\pi_*(R_0(3))
\]
where $H_*^{tmf}$ denotes relative homology $\pi_*(H\F_2\otimesover{tmf}-)$. 

\begin{proposition}\label{prop:relASSE2}
    There is a change of rings isomorphism
    \[E_2\cong \Ext_{\mathcal{E}(2)_*}(\F_2,H_*X)\]
    where $\mathcal{E}(2)_*=\F_2[\xi_1,\xi_2,\xi_3]/(\xi_1^2,\xi_2^2,\xi_3^2)$.
\end{proposition}
\begin{proof}
    One has a Kunneth isomorphism of $\mathcal{A}(2)_*$-comodules
\[
    H_*^{tmf}(R_0(3))\cong H_*^{tmf}(tmf_1(3))\otimesover{\F_2}H_*^{tmf}(tmf_0(3))
\]
The splittings $tmf_1(3)\simeq tmf\otimes DA(1)$ (see \cite{akhil}) and $tmf_0(3)\simeq tmf\otimes X$ result in isomorphisms of $\mathcal{A}(2)_*$-comodules
\begin{align*}
H_*^{tmf}(tmf_1(3))&\cong\mathcal{A}(2)_*\Boxover{E(2)_*}\F_2\cong\F_2[\xi_1^2,\xi_2^2]/(\xi_1^8,\xi_2^4)\\
H_*^{tmf}(tmf_0(3))&\cong H_*X
\end{align*}
from which the change of rings isomorphism follows.
\end{proof}

The $\mathcal{A}(2)_*$-comodule structure of $H_*X$ was given explicitly in \cite[Theorem 1.9]{CHR}, from which the $\mathcal{E}(2)_*$-comodule structure may be read off. A standard computation filtering by degrees yields the following.

\begin{proposition}\label{prop:relASSE2computation}
    The $E_2$-page of the relative Adams SS for $tmf_1(3)\otimesover{tmf}tmf_0(3)$ is given by
    \[E_2=\frac{\F_2[v_0,v_1,v_2]\{e_0,e_4,e_6,e_8,e_{10},e_{12}\}}{(v_0e_6=v_1e_4,v_0e_{10}=v_{2}e_4,v_1e_{10}=v_2e_6)}\]
    where $|v_i|=(2(2^i-1),1)$ and $|e_i|=(i,0)$. The spectral sequence collapses with no differentials as it is concentrated in even stems.
\end{proposition}

\subsection{The relative Hurewicz map and transfer classes} Proposition \ref{prop:relASSE2computation} implies that it suffices to find classes $e_4,e_6,e_8,e_{10},e_{12}$ in $\pi_*(R_0(3))$ of (relative) Adams filtration zero. Hence we must find certain classes that survive the relative Hurewicz map $\pi_*(R_0(3))\to H_*^{tmf}(R_0(3))$.

Since the fixed points of \(tmf_1(3)\) are a \(tmf\)-algebra, we have a map of \(C_2\)-equivariant naive \(E_\infty\)-rings \(i_\ast tmf\to tmf_1(3)\).

\begin{proposition}\label{prop:relhomologytmf03mackey}
    The Green functor $\underline{\pi}_*(i_*H\F_2\otimesover{i_*tmf}i_*tmf_1(3)\otimesover{i_*tmf}tmf_1(3))$ is isomorphic to $\F_2[u,v]/(u^4,v^2)\otimes-$ applied pointwise to the Green functor 
    \[\underline{R}_*:=\underline{\pi}_*(i_*H\F_2\otimesover{i_*tmf}tmf_1(3))\]
    where $|u|=2$ and $|v|=6$. The latter satisfies
    \begin{align*}
        \underline{R}_*(C_2/C_2)&\cong \F_2\{1,\epsilon_4,\epsilon_6,\epsilon_7,\epsilon_8,\epsilon_{10},\epsilon_{11},\epsilon_{12},\epsilon_{13},\epsilon_{14}\}\\
        \underline{R}_*(C_2/e)&\cong\F_2[\zeta_1^2,\zeta_2^2]/((\zeta_1^2)^4,(\zeta_2^2)^2)
    \end{align*}
    where the multiplication at $C_2/C_2$ is square zero in positive degrees. The restriction map in $\underline{R}_*$ is zero in positive degrees, and the nonzero transfers are as follows:
    \begin{align*}
        tr((\zeta_1^2)^2)&=\epsilon_4&
        tr((\zeta_1^2)^3)&=\epsilon_6&
        tr(\zeta_2^2)&=\epsilon_6\\
        tr(\zeta_1^2\zeta_2^2)&=\epsilon_8&
        tr((\zeta_1^2)^2\zeta_2^2)&=\epsilon_{10}&
        tr((\zeta_1^2)^3\zeta_2^2)&=\epsilon_{12}
    \end{align*}
\end{proposition}
\begin{proof}
    The first statement follows by naturality of the Kunneth formula. For the description of $\underline{R}_*$, we appeal to the computation of $\pi_\bigstar(i_*H\F_2\otimes tmf_1(3))$ in \cite[pg. 40]{CHR}. Each $\epsilon_i$ in positive degree is divisible by $\overline{v}_2$ and hence restricts to zero. The relations $a_\sigma^3\overline{v}_2=a_\sigma x_1^2\overline{v}_2=0$ along with $\ker(a_\sigma)=\mathrm{im}(tr)$ imply that $\epsilon_4,\epsilon_6,\epsilon_8,\epsilon_{10},\epsilon_{12}$ are in the image of the transfer. For degree reasons, the claimed transfers are the only possibilities, so it remains to explain why $tr((\zeta_1^2)^3)=tr(\zeta_2^2)$. 
    
    This follows from the relation $\ker(tr(-\cdot u_\sigma^{-1}))=\mathrm{im}(res)$, where $u_\sigma:{C_2}_+\wedge S^{1-\sigma}\to S^0$ is the map adjoint to the identity (see \cite{ArakiIriye}). This holds in the homotopy of any $C_2$-spectrum via the long exact sequence associated to the cofiber sequence ${C_2}_+\to S^0\xrightarrow{a_\sigma} S^\sigma$. The claim then follows from the fact that the class $x_3$ from \cite[Definition 4.7]{CHR} restricts to $u_\sigma\xi_2^2=u_\sigma(\zeta_1^6+\zeta_2^2)$.
\end{proof}

The Mackey structure of Proposition \ref{prop:relhomologytmf03mackey} will allow us to find classes in $\pi_*(R_0(3))$ of Adams filtration zero. In particular, since we have a morphism of Mackey functors
\[\underline{\pi}_*(i_*tmf_1(3)\otimesover{i_*tmf}tmf_1(3))\to \underline{\pi}_*(i_*H\F_2\otimesover{i_*tmf}i_*tmf_1(3)\otimesover{i_*tmf}tmf_1(3))\]
it suffices to find lifts of $(\zeta_1^2)^2$, $(\zeta_1^2)^3$, $\zeta_1^2\zeta_2^2$, $(\zeta_1^2)^2\zeta_2^2$, and $(\zeta_1^2)^3\zeta_2^2$ along the underlying relative Hurewicz map
\begin{equation}\label{eq:relhurewicz}
    \pi_*(tmf_1(3)\otimesover{tmf}tmf_1(3))\to H_*^{tmf}(tmf_1(3)\otimesover{tmf}tmf_1(3))
\end{equation}
This map is dual to the canonical projection
\[
\begin{tikzcd}
\Spec(\F_2)\timesover{\mathcal{M}_{cub}}\Spec(\Z[a_1,a_3])\timesover{\mathcal{M}_{cub}}\Spec(\Z[a_1,a_3])\arrow[d,"\pi"]\\
\Spec(\Z[a_1,a_3])\timesover{\mathcal{M}_{cub}}\Spec(\Z[a_1,a_3])
\end{tikzcd}
\]
where the maps $\Spec(\F_2)\to \mathcal{M}_{cub}$ and $\Spec(\Z[a_1,a_3])\to \mathcal{M}_{cub}$ classify the cubic curves $y^2=x^3$ and $y^2+a_1xy+a_3=x^3$, respectively. The Kunneth isomorphism
\[H_*^{tmf}(tmf_1(3))\otimes_{\F_2}H_*^{tmf}(tmf_1(3))\to H_*^{tmf}(tmf_1(3)\otimesover{tmf}tmf_1(3))\]
is dual to the morphism of stacks
\[
\begin{tikzcd}    
\Spec(\F_2)\timesover{\mathcal{M}_{cub}}\Spec(\Z[a_1,a_3])\timesover{\mathcal{M}_{cub}}\Spec(\Z[a_1,a_3])\arrow[d,"\phi"]\\
\bigg(\Spec(\F_2)\timesover{\mathcal{M}_{cub}}\Spec(\Z[a_1,a_3])\bigg)\timesover{\Spec(\F_2)}\bigg(\Spec(\F_2)\timesover{\mathcal{M}_{cub}}\Spec(\Z[a_1,a_3])\bigg)
\end{tikzcd}
\]
which, on the left factor is given by projection away from the third factor of the source, and on the right factor is given by projection away from the second factor of the source.

Therefore, the relative Hurewicz map followed by the inverse of the Kunneth isomorphism gives the map
\begin{equation}\label{eq:relhurewicz}
    (\Z[a_1,a_3,s,t]/\sim)\to \F_2[u,v]/(u^4,s^2)\tens{\F_2}\F_2[\zeta_1^2,\zeta_2^2]/((\zeta_1^2)^2,(\zeta_2^2)^2)
\end{equation}
dual to the morphism of stacks $\pi\circ\phi^{-1}$, where $s,t$ are the parameters of the universal isomorphism 
\[(y^2+a_1xy+a_3=x^3)\to(y^2+\eta_R(a_1)xy+\eta_R(a_3)=x^3)\]
$u,v$ are the parameters of the universal isomorphism
\[(y^2=x^3)\to(y^2+\eta_R(a_1)xy+\eta_R(a_3)=x^3)\]
in the left factor, and $\zeta_1^2,\zeta_2^2$ are the parameters of the same isomorphism in the right factor. Note that this agrees with the choice of generators of the same names in $\mathcal{A}(2)_*$ and in particular those used in Proposition \ref{prop:relhomologytmf03mackey} because the map
\[\pi_*(tmf_1(3)\otimesover{tmf}tmf_1(3))\to\pi_*(H\F_2\otimesover{tmf}H\F_2)=\mathcal{A}(2)_*\]
sends $s\mapsto\zeta_1^2$ and $t\mapsto\zeta_2^2$.

\begin{proposition}\label{prop:relhurewicz}
    The relative Hurewicz map of (\ref{eq:relhurewicz})
   under the identifications of (\ref{eq:relhurewicz}) sends $a_i\mapsto0$ and sends
    \begin{align*}
        s&\mapsto u+\zeta_1^2,&
        t&\mapsto u^3+v+u(\zeta_1^2)^2+\zeta_2^2.
    \end{align*}
\end{proposition}
\begin{proof}
    Following the discussion above, one sees that the ring map dual to $\pi$ sends $s\mapsto s$ and $t\mapsto t$. It suffices then to show that the ring map dual to $\phi$ sends $u\mapsto u$, $v\mapsto v$, $\zeta_1^2\mapsto u+s$, and $\zeta_2^2\mapsto us^2+v+t$. This follows by computing the translation parameters of the composite
    \[(y^2=x^3)\to(y^2+a_1xy+a_3y=x^3)\to (y^2+\eta_R(a_1)xy+\eta_R(a_3)y=x^3)\]
\end{proof}

\begin{corollary}\label{cor:transferlifts}
    The classes
    \[tr(s^2),tr(s^3),tr(t),tr(st),tr(s^2t),tr(s^3t)\in \pi_*(R_0(3))\]
    have relative Adams filtration zero.
\end{corollary}
\begin{proof}
    Let $h$ denote the relative Hurewicz map of (\ref{eq:relhurewicz}). The formulas of Propositions \ref{prop:relhurewicz} and \ref{prop:relhomologytmf03mackey} imply that the transfers
\[tr(h(s^2)),tr(h(s^3)),tr(h(t)),tr(h(st)),tr(h(s^2t)),tr(h(s^3t))\]
    are all nonzero, so by Mackey functoriality the claimed transfers are nonzero in relative homology.
\end{proof}

\subsection{The underlying $C_2$-action and the image of the restriction map} Proposition \ref{prop:relASSE2computation} tells us that the image of the restriction map of Corollary \ref{cor:relativeresinjective} is the sub-$A$-module generated by $1,e_4,e_6,e_8,e_{10},e_{12}$, where $e_i$ is any class in $\pi_i$ of relative Adams filtration zero, and Corollary \ref{cor:transferlifts} gives us preferred classes of Adams filtration zero. It remains to determine the restriction of the classes in \ref{cor:transferlifts}. The Mackey formula $res\circ tr=1+\gamma$ reduces this to determining the underlying $C_2$-action.

\begin{proposition}\label{prop:underlyingC2action}
    Let $\gamma$ denote the generator of $C_2$. The action of $C_2$ on 
    \[\pi_*(tmf_1(3)\otimesover{tmf}tmf_1(3))=\underline{\pi}_*(i_*tmf_1(3)\otimesover{i_*tmf}tmf_1(3))(C_2/e)\]
fixes $a_1$ and $a_3$ and satisfies $\gamma(s)=s-\eta_R(a_1)$ and $\gamma(t)=t-\eta_R(a_3)$.
\end{proposition}
\begin{proof}
    The classes $a_1$ and $a_3$ are fixed because they are in the image of the equivariant map
    \[\eta_L:i_*tmf_1(3)\to i_*tmf_1(3)\otimesover{i_*tmf}tmf_1(3)\]
    and $C_2$ acts trivially on the source. The equivariant map
    \[\eta_R:tmf_1(3)\to i_*tmf_1(3)\otimesover{i_*tmf}tmf_1(3)\]
    sends $a_1\mapsto\eta_R(a_1)$ and $a_3\mapsto\eta_R(a_3)$, and $\gamma(a_1)=-a_1$ and $\gamma(a_3)=-a_3$ in the source (see \cite[Proposition 3.4]{mahowaldrezk})). The formulas
    \begin{align*}
        \eta_R(a_1)&=a_1+2s & \eta_R(a_3)&=a_3+\frac{1}{3}a_1s^2+\frac{1}{3}a_1^2s+2t
    \end{align*}
    along with the fact that $\pi_*(tmf_1(3)\otimesover{tmf}tmf_1(3))$ is torsion free now imply the claimed formulas for $\gamma(s)$ and $\gamma(t)$.
\end{proof}

We now can determine the restriction of each of the transfers of Corollary $\ref{cor:transferlifts}$. In the following, let $(A,\Gamma)=(\Z[a_1,a_3],A[s,t]/\sim)$ denote the $2$-local Weierstrass algebroid.

\begin{theorem}\label{thm:tmf03computation}
    There is an isomorphism of left $(A,\Gamma)$-comodules
     \begin{align*}
     \pi_*R_0(3)&=\pi_*(tmf_1(3)\otimesover{tmf}tmf_0(3))\cong \frac{A\{1,e_4,e_6,e_8,e_{10},e_{12}\}}{(a_1e_4=2e_6,a_3e_4=2e_{10},a_3e_6=a_1e_{10})}\\
        \psi(e_4)&=1\otimes e_4+e_4\otimes 1\\
        \psi(e_6)&=1\otimes e_6-s\otimes e_4+e_6\otimes 1\\
        \psi(e_8)&=1\otimes e_8+2s\otimes e_6-\frac{s^2}{3}\otimes e_4+e_8\otimes 1\\
        \psi(e_{10})&=1\otimes e_{10}+(s^2-\frac{sa_1}{3})\otimes e_6+(t-\frac{s^3}{3})\otimes e_4+e_{10}\otimes 1\\
        \psi(e_{12})&=1\otimes e_{12}+5s\otimes e_{10}+(3s^2+3a_1s)\otimes e_8+(5t+\frac{44}{9}s^3+\frac{16}{9}a_1s^2+\frac{2}{3}a_1^2s)\otimes e_6\\&+(a_1s^3-3a_3s-3a_1t-11st)\otimes e_4+e_{12}\otimes 1
    \end{align*}
    Moreover, $\pi_*R_0(3)$ is the sub $(A,\Gamma)$-comodule algebra of $\Gamma$ generated by the classes
    \begin{align*}
        e_4&:=2s^2+2sa_1\\
        e_6&:=a_1^2s+a_1s^2\\
        e_8&:=2st+a_1t+a_3s+\frac{a_1^2s^2}{3}+\frac{a_1s^3}{3}\\
        e_{10}&:=a_1a_3s+a_3s^2\\
        e_{12}&:=2s^3t-a_1^2a_3s+a_3s^3+a_1^3t+4a_1^2st+5a_1s^2t
    \end{align*}

\end{theorem}
\begin{proof}
    Using the action formulas of Proposition \ref{prop:underlyingC2action}, we compute 
    \begin{align*}
        res(tr(s^2))&=a_1^2+2a_1s+2s^2\\
        res(tr(s^3))&=-a_1^3-3a_1^2s-3a_1s^2\\
        res(tr(t))&=-a_3-\frac{a_1^2s}{3}-\frac{a_1s^2}{3}\\
        res(tr(st))&=a_1a_3+\frac{a_1^3s}{3}+a_3s+\frac{2a_1^2s^2}{3}+\frac{a_1s^3}{3}+a_1t+2st\\
        res(tr(s^2t))&=-a_1^2a_3-\frac{a_1^4s}{3}-3a_1a_3s-a_1^3s^2-a_3s^2-\frac{2a_1^2s^3}{3}-2a_1^2t-4a_1st\\
        res(tr(s^3t))&=a_1^3a_3+\frac{a_1^5s}{3}+6a_1^2a_3s+\frac{4a_1^4s^2}{3}+4a_1a_3s^2+a_1^3s^3+a_3s^3+4a_1^3t\\
        &+10a_1^2st+5a_1s^2t+2s^3t
    \end{align*}
    Note that $a_1$ and $a_3$ are in the image of the restriction and are in strictly positive Adams filtration. Subtracting off terms of the form $p(a_i)res(x)$ therefore still gives a class of Adams filtration zero, and this yields the above simplifications. The $A$-module relations and coactions are computed in $\Gamma$ using the fact that the restriction map is injective.
\end{proof}

\printbibliography

\end{document}